\documentclass[11pt]{amsart}

\usepackage[numbers,sort]{natbib}

\usepackage{amsfonts}
\usepackage{amsmath}
\usepackage{amssymb}
\usepackage{amscd}
\usepackage{color}
\usepackage{xcolor}
\usepackage{amsthm}
\usepackage[hmargin=4cm,vmargin=4cm]{geometry}
\usepackage{parskip}

\usepackage{enumitem}

\newtheorem{thm}{Theorem}

\newtheorem{thm*}{Theorem}
\newtheorem{prop}{Proposition}
\newtheorem{lma}[prop]{Lemma}

\newtheorem{cor}[prop]{Corollary}

\theoremstyle{definition}

\theoremstyle{remark}

\newtheorem{rmk}[prop]{Remark} 

\newcommand{\F}{{\mathbb{F}}}
\newcommand{\R}{{\mathbb{R}}}
\newcommand{\Z}{{\mathbb{Z}}}
\newcommand{\C}{{\mathbb{C}}}

\newcommand{\D}{{\mathbb{D}}}
\newcommand{\bK}{{\mathbb{K}}}
\newcommand{\bF}{{\mathbb{F}}}

\newcommand{\del}{\partial}

\newcommand{\sm}[1]{C^\infty(#1)}

\newcommand{\K}{\mathcal{K}}

\newcommand{\cL}{\mathcal{L}}

\newcommand{\til}[1]{\widetilde{#1}}
\newcommand{\wh}[1]{\widehat{#1}}
\newcommand{\ol}[1]{\overline{#1}}

\newcommand{\zt}{{\Z/(2)}}

\DeclareMathOperator{\ima}{\mathrm{im}}

\newcommand{\tmin}{{\text{min},\bK}}

\newcommand{\om}{\omega}

\newcommand{\eps}{\epsilon}

\newcommand{\cA}{\mathcal{A}}

\newcommand{\cD}{\mathcal{D}}
\newcommand{\cE}{\mathcal{E}}

\newcommand{\cH}{\mathcal{H}}

\newcommand{\cJ}{\mathcal{J}}

\newcommand{\cO}{\mathcal{O}}

\newcommand{\cP}{\mathcal{P}}

\newcommand{\fix}{\mathrm{Fix}}

\newcommand{\pr}{{pseudo-rotation}}

\renewcommand{\hat}{\wh}

\def\mrm#1{{\mathrm{#1}}}

\def\cl#1{{\mathcal{#1}}}

\DeclareMathOperator{\Ham}{\mathrm{Ham}}

\DeclareMathOperator{\id}{\mathrm{id}}
\DeclareMathOperator{\spec}{\mathrm{Spec}}
\DeclareMathOperator{\Spec}{\mathrm{Spec}}
\DeclareMathOperator{\loc}{\mathrm{loc}}

\def\Hk{H^{(k)}}

\def\H2{H^{(2)}}

\begin{document}

\title{Pseudo-rotations and Steenrod squares}
\author{Egor Shelukhin}
\date{}

\begin{abstract}
In this note we prove that if a closed monotone symplectic manifold $M$ of dimension $2n,$ satisfying a homological condition that holds in particular when the minimal Chern number is $N>n,$ admits a Hamiltonian \pr, then the quantum Steenrod square of the point class must be deformed. This gives restrictions on the existence of {\pr}s. Our methods rest on previous work of the author, Zhao, and Wilkins, going back to the equivariant pair-of-pants product-isomorphism of Seidel.
\end{abstract}






\maketitle

\section{Introduction}


This paper deduces obstructions, in terms of pseudo-holomorphic curves, to the existence of Hamiltonian pseudo-rotations. The latter are special Hamiltonian diffeomorphisms of symplectic manifolds, characterized by having the minimal possible number of periodic points, of all integer periods. 

The classical notion of pseudo-rotations of the two-sphere, or the two-disk, has appeared in the work of Anosov and Katok \cite{AnosovKatok} (see also Fayad and Katok \cite{FayadKatok}) and was further investigated extensively in the field of conservative dynamics. Indeed, the simplest example of a pseudo-rotation is an irrational point in the Hamiltonian $S^1 = \R/\Z$ action on $S^2$ with Hamiltonian (a constant multiple of) the standard height function. However, \cite{AnosovKatok, FayadKatok} construct, by means of the conjugation method, examples of pseudo-rotations of $S^2$ with different dynamical properties than those of such irrational rotations. For example they admit precisely three ergodic measures: two fixed points, and the area measure. 


Recent years saw a renewed interest in pseudo-rotations considered from the point of view of symplectic rigidity phenomena. For instance, Barney Bramham proved in \cite{BramhamApprox,BramhamRig} that all Hamiltonian pseudo-rotations of the closed disk can be $C^0$ approximated by a sequence of conjugates of rational rotations, and that they are $C^0$ rigid in a suitable sense, provided that their rotation number is sufficiently Liouvillean. In the recent seminal paper \cite{GG-pseudorotations} by Ginzburg and G\"{u}rel, the $C^0$-rigidity result of Bramham, as well as other results regarding the dynamics of Hamiltonian pseudo-rotations, were established for complex projective spaces of all dimensions.

This paper, as well as \cite{CGG}, takes a different point of view, considering pseudo-rotations to be strong counter-examples to the Conley conjecture. From this perspective, a conjecture of Chance and McDuff, arising from \cite{McDuff-uniruled}, asserts that the existence of such counter-examples, and hence that of pseudo-rotations, must imply the existence of non-trivial algebraic counts of pseudo-holomorphic spheres in the manifold. Here we provide an instance of such an implication, ruling out in particular the existence of pseudo-rotations on a closed monotone symplectic manifold of dimension $2n$ with minimal Chern number $n+2$ or greater. Further extensions of our results have very recently appeared in \cite{CGG2, PRQSR}.

\section{Setup and main results}

In this paper, unless otherwise specified, we work with a closed monotone symplectic manifold $(M,\om)$ of dimension $2n,$ and rescale the symplectic form so that $[\om] = \kappa \cdot c_1(TM)$ on the image of the Hurewicz map $\pi_2(M) \to H_2(M;\Z)$ for $\kappa = 2.$ Recall that the minimal Chern number of $(M,\om)$ is the index \[N = N_M = [\Z:\mrm{image}(c_1(TM): \pi_2(M) \to \Z)].\] For a Hamiltonian diffeomorphism $\phi \in \Ham(M,\om),$ we denote by $\fix(\phi)$ the set of {\em contractible} fixed points of $\phi,$ and by $x^{(k)}$ for $x \in \fix(\phi)$ its image under the inclusion $\fix(\phi) \subset \fix(\phi^k).$ Contractible means that the homotopy class of the path $\alpha(x,\phi) = \{ \phi^t_H(x)\}$ for a Hamiltonian $H \in \sm{[0,1] \times M, \R}$ generating $\phi$ as the time-one map $\phi^1_H = \phi$ of its Hamiltonian flow, is trivial. This class does not depend on the choice of Hamiltonian by a classical argument in Floer theory.

We say that a Hamiltonian diffeomorphism $\phi \in \Ham(M,\om)$ is a $\bK = \F_2$ {\em Hamiltonian \pr} if:
\begin{enumerate}[label = (\roman*)]
\item \label{cond: perfect} It is perfect, that is for all iterations $k \geq 1$ of $\phi,$ $\fix(\phi^k) = \fix(\phi)$ is finite. In other words, $\phi$ admits no simple periodic orbits of order $k >1.$
\item \label{cond: dim loc positive} For each $x \in \fix(\phi),$ the dimension of the local Floer homology of $\phi$ at $x$ satisfies $\dim_{\bK} HF^{\loc}(\phi,x) \geq 1,$ and furthermore,

\item \label{cond: sum of Betti} $N(\phi^k,\bK) = \sum_{x \in \fix(\phi)} \dim_{\bK} HF^{\loc}(\phi^k,x^{(k)}) = \dim_{\bK} H_*(M)$ for all $k\geq 1.$
\end{enumerate}

\begin{rmk}\label{rmk: pseudorot}
We observe that a perfect Hamiltonian diffeomorphism necessarily has no symplectically degenerate maxima (see \cite{GG-revisited}). Furthermore, if all the points in $\fix(\phi^k)$ are non-degenerate, for all $k \geq 1,$ then condition \ref{cond: dim loc positive} is automatically satisfied, and all iterations are {\em admissible}, that is $\lambda^k \neq 1$ for all eigenvalues $\lambda \neq 1$ of $D(\phi)_x$. Furthermore, by the Smith inequality in local Floer homology \cite{CineliGinzburg, SZhao-pants}, conditions \ref{cond: perfect} and \ref{cond: sum of Betti} imply, for iterations of the form $k = 2^m,$ the stronger statement that for all $x \in \fix(\phi),$ $\dim_{\bK} HF^{\loc}(\phi^k,x^{(k)}) = \dim_{\bK} HF^{\loc}(\phi, x).$ Moreover, \cite[Theorem A]{S-HZ} suggests that when a Hamiltonian diffeomorphism has a finite number of periodic points, then a condition like \ref{cond: sum of Betti} should be satisfied. Showing this would bridge the gap between the initial Chance-McDuff conjecture (see for example \cite{GG-revisited}) and the main result of this note, Theorem \ref{thm: uniruled}. Finally, we include condition \ref{cond: dim loc positive} for compatibility with the literature: we do not use it below. We refer to \cite{GG-pseudorotations} for further discussion of dynamics of Hamiltonian {\pr}s in higher dimensions.
\end{rmk}


We call a symplectic manifold strongly uniruled if there exists a non-trivial three point genus-zero Gromov-Witten invariant $\left< [pt], a, b\right>_{\beta},$ for $\beta \in H_2(M,\Z) \setminus \{0\}.$ By \cite[Lemma 2.1]{McDuff-uniruled}, $(M,\om)$ is not strongly uniruled if and only if the $\Lambda$-linear subspace \[\cl Q_{-} = H_{\ast <2n}(M) \otimes\Lambda \subset QH(M,\Lambda),\] where $\Lambda$ is the minimal Novikov field of $(M,\om),$ {with quantum variable $q$ of degree $(-2N),$} is an ideal in the quantum homology ring $QH(M,\Lambda).$ Recall that the quantum product in $QH(M,\Lambda)$ is a deformation of the intersection product on homology given by three-point genus-zero Gromov-Witten invariants. Alternatively $[pt] \ast r = 0$ for all $r \in \cl Q_{-}.$ Note that in this case, $[pt] \ast [pt] = 0$ in particular. 


A generally different stronger notion than $[pt] \ast [pt] = 0,$ is that the quantum Steenrod square  $\cl{QS}([pt]),$ defined in \cite{Wilkins}, of the point class satisfy \[ \cl{QS}([pt]) = h^{2n} [pt].\] This corresponds to there being no quantum corrections when passing from the classical Steenrod square of the point class on $M$ to its quantum version. In this case we say that $M$ is not {\em $\zt$-Steenrod uniruled}.

\begin{rmk}\label{rmk: Steenrod basic}
Observe that when $(M,\om)$ is $\zt$-Steenrod uniruled, then by a Gromov compactness argument there exists a $J$-holomorphic curve through each point of $M.$ Furthermore, recall from \cite{Wilkins-PSS} that setting the quantum variable to be of cohomological degree $2N,$ $h$ to be of degree $1,$ and considering cohomological degree on the homology classes, $\cl{QS}([pt])$ must be of degree $2 \deg([pt]) = 4n.$ Hence if $N > 2n = \dim(M),$ then $(M,\om)$ is automatically not $\zt$-Steenrod uniruled. By the same token, if $N=2n,$ then being $\zt$-Steenrod uniruled is equivalent to $[pt] \ast [pt] \neq 0.$ In fact, in this case $[pt] \ast [pt] = q [M].$ We note that while the above choice of degrees was convenient for this remark, throughout the paper we work with homology and use different conventions for degrees.
\end{rmk} 


\begin{rmk}\label{rmk: Steenrod fancy}
The main result of \cite{SeidelWilkins} implies that if $N>n,$ and $[\om]$ lies in the lattice $H^2(M,\Z)/{\mrm{torsion}} \subset H^2(M,\R),$ then $(M,\om)$ being $\zt$-Steenrod uniruled implies that $(M,\om)$ is strongly uniruled. In fact, in this case there exists a Gromov-Witten invariant \[\left<[pt],[pt],D\right>_{\beta}\] that does not vanish modulo $2$ for a suitable divisor class $D \in H_{2n-2}(M;\Z).$ Note that when this holds, by a degree count, we obtain $N = n+1.$ Therefore Theorem \ref{thm: uniruled} implies that there is no $\bF_2$ Hamiltonian {\pr } for $N>n+1.$ Under the additional assumption that $\left< [\om]^n, [M] \right>$ is odd, one may prove the above statement by a straightforward adaptation of the proof of \cite[Lemma 6.1]{Wilkins}.


\end{rmk}

Finally, we say that $(M,\om)$ satisfies the Poincar\'{e} duality property, if for all $\til{\phi} \in \til{\Ham}(M,\om),$ the following Poincar\'{e} duality identity of Hamiltonian spectral invariants (defined in Section \ref{subsec:spec} below) holds: \[ c([M], \til{\phi}^{-1}) = - c([pt], \til{\phi}).\] It is well-known (see \cite{EntovPolterovichCalabiQM}) that $(M,\om)$ with $N > n$ satisfy this property. The main result of this note is the following.

\medskip

\begin{thm}\label{thm: uniruled}
Let $(M,\om)$ be a closed monotone symplectic manifold satisfying the Poincar\'{e} duality property, and admitting an $\F_2$ Hamiltonian {\pr } $\phi.$ Then $(M,\om)$ is $\zt$-Steenrod uniruled. 
\end{thm}


In view of Remarks \ref{rmk: Steenrod basic} and \ref{rmk: Steenrod fancy} above we conclude the following result pertaining to the Chance-McDuff conjecture.

\begin{cor}\label{cor: CM}
Let $(M,\om)$ be a closed monotone symplectic manifold with $N>n.$ If $N>n+1$ then $(M,\om)$ does not admit $\F_2$ Hamiltonian \pr s. If $N=n+1,$ and $(M,\om)$ admits an $\F_2$ Hamiltonian \pr, then $(M,\om)$ satisfies $[pt] \ast [pt] \neq 0,$ and in particular it is strongly uniruled. 
\end{cor}

\medskip
\begin{rmk}
The author was made aware that new relations between {\pr}s and holomorphic curves were also found in recent work of \c{C}ineli, Ginzburg, and G\"{u}rel \cite{CGG}. 
\end{rmk}

\begin{rmk}
The symplectic manifold $(\C P^n, \om_{st})$ has $N=n+1,$ and verifies the hypothesis and the conclusion of the theorem separately. Formally speaking, this result seems to be generally new even when $\phi$ comes as the irrational rotation with respect to a Hamiltonian $S^1$-action with isolated fixed points (however it does not seem to give new examples in that case: see \cite{Sabatini-etal, Sabatini-circle, Charton-masters}). We remark that by a result of McDuff \cite{McDuff-uniruled} all Hamiltonian $S^1$-manifolds are uniruled, the latter being defined with $m$-point genus $0$ Gromov-Witten invariants with arbitrary $m\geq 3.$
\end{rmk}

The strategy of the proof of the main result is the direct comparison between the following two results. First, the following Lusternik-Shnirelman\footnote{also transcribed from Russian as  Lyusternik, Lusternick, Ljusternik, and respectively Schnirelmann, Shnirel'man.} type result was shown in \cite{GG-revisited}. Define for $a \in QH(M),$ $\til{\phi} \in \til{\Ham}(M,\om)$ the asymptotic spectral invariant by \[\ol{c}(a,\til{\phi}) = \lim_{k \to \infty} \frac{1}{k} c(a, \til{\phi}^k).\] Recall that a  symplectic manifold is called {\em rational} if the period group of the symplectic form is a discrete subgroup of $\R.$

\begin{thm}\label{thm: LS}
Let $(M,\om)$ be rational, and $\bK$ be a ground field. Suppose $\phi$ has isolated periodic points of each period, none of which is a symplectically degenerate maximum. Then for each $k \geq 1,$ and lift $\til{\phi}$ of $\phi$ to $\til{\Ham}(M,\om),$ \[\frac{1}{k} c([M],  \til{\phi}^k) > \ol{c}([M],\til{\phi}).\]
\end{thm}

We observe that in particular, when $\phi$ is perfect, by \cite[Theorem 1.18]{GG-ai}, none of the fixed points of $\phi$ are symplectically degenerate maxima, and hence Theorem \ref{thm: LS} applies. Hence, in the case when $(M,\om)$ satisfies the Poincar\'{e} duality property, we obtain by applying Theorem \ref{thm: LS} to $\til{\phi}^{-1},$ for $\phi$ perfect, that \[\frac{1}{k} c([pt],  \til{\phi}^k) < \ol{c}([pt],\til{\phi}).\]

In particular there exists $m \geq 1,$ such that for $\til{\psi} = \til{\phi}^{2^m},$ \begin{equation}\label{eq: smaller} c([pt],  \til{\psi}^2) > 2 \cdot {c}([pt],\til{\psi}).\end{equation}

Second, we prove below the following statement. 

\begin{thm}\label{thm: pseudo non eq uni}
Let $\psi$ be an $\F_2$ Hamiltonian {\pr } on $(M,\om)$ that is not $\zt$-Steenrod uniruled.  
Then \begin{equation}\label{eq: greater or equal} c([pt],\til{\psi}^2) \leq 2 \cdot c([pt], \til{\psi})\end{equation} for each $\til{\psi} \in \til{\Ham}(M,\om)$ covering $\psi.$
\end{thm}

The proof of Theorem \ref{thm: uniruled} is now the combination of \eqref{eq: smaller} and \eqref{eq: greater or equal}.

\section*{Acknowledgements}
I thank Viktor Ginzburg, Ba\c{s}ak G\"{u}rel, and Nicholas Wilkins, for useful conversations. I thank Dusa McDuff for explaining to me her work on Hamiltonian $S^1$-manifolds. Finally, I thank the referees for useful comments and suggestions that lead to improvements in the exposition. At the University of Montr\'{e}al I was supported by an NSERC Discovery Grant and by the Fonds de recherche du Qu\'{e}bec - Nature et technologies.

\section{Preliminaries}\label{subsec:prelim}
We present general preliminaries on Hamiltonian Floer homology and spectral invariants, to be applied in two settings below: usual and equivariant. To keep the paper short, we refer to \cite{Seidel-pants, Wilkins, Wilkins-PSS, SZhao-pants, S-HZ} for preliminaries related to equivariant Floer homology, related natural operations and structures, and to \cite{OhBook} for further details on Hamiltonian Floer homology.

We denote by $\cH \subset C^{\infty}([0,1] \times M,\R)$ the space of time-dependent Hamiltonians on $M,$ that vanish near $0$ and $1,$ where $H_t(-) = H(t,-)$ is normalized to have zero mean with respect to $\om^n.$ We shall consider $H \in \cl{H}$ as a $1$-periodic function on $\R \times M$ in the $\R$-coordinate. The time-one maps of isotopies $\{\phi^t_H\}_{t \in [0,1]}$ generated by time-dependent vector fields $X^t_H,$ such that $\iota_{X^t_H} \omega = - d(H_t),$ are called Hamiltonian diffeomorphisms and form the group $\Ham(M,\om).$ For $H \in \cH$ we call $\overline{H} \in \cH$ the Hamiltonian $\overline{H}(t,x) = - H(t,\phi^t_H x)$ generates the flow assigning $t \in [0,1]$ the map $\phi^t_{\overline{H}} = (\phi^t_H)^{-1},$ and for $F,G \in \cl{H}$ the Hamiltonian $F \# G(t,x) = F(t,x) + G(t, (\phi^t_F)^{-1} x)$ generates the path $\{\phi^t_F \phi^t_G\}_{t\in [0,1]}.$ Recall that homotopic Hamiltonian isotopies relative to endpoints give naturally isomorphic graded filtered Floer complexes. The universal cover $\til{\Ham}(M,\om)$ is constructed as the space of such homotopy classes, and carries a natural group structure. Further, for $H \in \cl{H}$ we set $\Hk(t,x) = k H(kt,x) \in \cl{H}.$ If the flow of $H$ generates $\til{\phi} \in \til{\Ham}(M,\om),$ then $\Hk$ generates $\til{\phi}^k.$

Finally, let $\cJ(M,\om)$ be the space of $\om$-compatible almost complex structures on $M.$

\subsection{Hamiltonian Floer homology.} \label{subsec:abs-Ham}

Consider $H \in \cH.$ Let $\cL_{pt} M$ be the space of contractible loops in $M.$ Let $c_M: \pi_1(\cL_{pt} M) \cong \pi_2(M) \to 2 N_M \cdot \Z,$ be the surjection given by $c_M(A) = 2 \left< c_1(M,\om), A\right>.$ Let $\til{\cL}^{\min}_{pt} M = \til{\cL_{pt}} \times_{c_M} (2 N_M \cdot \Z)$ be the cover of $\cL_{pt} M$ associated to $c_M.$ The elements of $\til{\cL}^{\min}_{pt} M$ can be considered to be equivalence classes of pairs $(x,\overline{x})$ of $x \in {\cL}_{pt} M$ and its capping $\overline{x}:\D \to M,$ $\overline{x}|_{\del \D} = x.$ Of course $x$ is determined by $\ol{x}.$ The symplectic action functional \[\cA_{H}: \til{\cL}^{\min}_{pt} M \to \R \] is given by \[\cA_{H}(\overline{x}) = \int_0^1 H(t,x(t)) - \int_{\overline{x}} \om,\] that is well-defined by monotonicity: $[\om]= \kappa \cdot c_M.$ Assuming that $H$ is non-degenerate, that is the graph $\mrm{graph}(\phi^1_H) = \{(\phi^1_H(x),x)\,|\, x \in M\}$ intersects the diagonal $\Delta_M \subset M \times M$ transversely, the generators over the base field $\bK$ of the Floer complex $CF(H;J)$ are the lifts $ \til{\cO}(H)$ to $\til{\cL}^{\min}_{pt} M$ of $1$-periodic orbits $\cO(H)$ of the Hamiltonian flow $\{\phi^t_H\}_{t \in [0,1]}.$ These are the critical points of $\cA_{H},$ and we denote by $\spec(H) = \cA_H(\til{\cO}(H))$ the set of its critical values. Choosing a generic time-dependent $\om$-compatible almost complex structure $\{J_t \in \cJ(M,\om)\}_{t \in [0,1]},$ and writing the asymptotic boundary value problem on maps $u:\R \times S^1 \to M$ defined by the negative formal gradient on $\cL_{pt} M$ of $\cA_{H},$ the count of isolated solutions with signs determined by a suitable orienation scheme, modulo $\R$-translation, gives a differential $d_{H;J}$ on the complex $CF(H;J),$ $d^2_{H;J} = 0.$ This complex is graded by the Conley-Zehnder index $CZ(H,\bar{x})$ \cite{Salamon-lectures, SalamonZehnder}. The Conley-Zehnder index has the property that the action of the generator $A = 2N_M$ of $2N_M\cdot \Z$ has the effect $CZ(H,\bar{x} \# A) = CZ(H,\bar{x}) - 2N_M,$ and it is normalized to be equal to $n$ at a maximum of a small autonomous Morse Hamiltonian. Its homology $HF_*(H)$ does not depend on the generic choice of $J.$ Moreover, considering generic families interpolating between different Hamiltonians $H,H',$ and writing the Floer continuation map, where the negative gradient depends on the $\R$-coordinate we obtain that $HF_*(H)$ in fact does not depend on $H$ either. While $CF_*(H,J)$ is finite-dimensional in each degree, it is worthwhile to consider its completion in the direction of decreasing action. In this case it becomes a free graded module of finite rank over the Novikov field \[\Lambda_{\bK} = \Lambda_{M,\tmin} = \bK[q^{-1},q]]\] with $q$ being a variable of degree $(-2N_M).$ This field carries a non-Archimedean valuation $\nu: \Lambda_{\bK} \to \R \cup \{+\infty\}$ given by $\nu(0) = \infty,$ and \[ \nu(\sum a_j q^{j}) = j_0 \cdot A_M,\]  where $A_M = \kappa \cdot N_M,$ and $j_0 = \min \{ j\;|\; a_j \neq 0 \}.$ It satisfies the properties:

\begin{enumerate}
	\item $\nu(x) = +\infty$ if and only if $x = 0,$
	\item $\nu(xy) = \nu(x) + \nu(y)$ for all $x,y \in \Lambda_{\bK},$
	\item $\nu(x+y) \geq \min\{\nu(x),\nu(y)\},$ for all $x,y \in \Lambda_{\bK}.$
\end{enumerate}

Moreover, we extend $\cl A_H$ to a non-Archimedean filtration on $CF(H,J)$ by \begin{equation}\label{eq: action of sum} \cl A_H( \sum \lambda_j \ol{x}_j ) = \max\{{ -\nu(\lambda_j)} + \cl A_H(\ol{x}_j) \},\end{equation} for a $\Lambda_{\bK}$-basis of $CF(H,J)$ consisting of capped orbits $\ol{x}_j \in \til{\cl O}(H).$ 

Recall, following \cite{UsherZhang}, that for a field $\Lambda$ with non-Archimedean valuation $\nu,$ a function $l:C \to \R \cup \{-\infty\}$ on a finite-dimensional $\Lambda$-module $C,$ is called a non-Archimedean filtration (function), if it satisfies the following properties: \begin{enumerate}
	\item $l(x) = -\infty$ if and only if $x = 0,$
	\item $l(\lambda x) = l(x) - \nu(\lambda)$ for all $\lambda \in \Lambda, x \in C,$
	\item \label{prop:maximu} $l(x+y) \leq \max\{l(x),l(y)\},$ for all $x,y \in C.$
\end{enumerate}

We call a complex $(C,d)$ with $C$ a finite-dimensional $\Lambda$-module with a non-Archimedean filtration $l$ {\em filtered} if $l(dy) \leq l(y)$ for all chains $y$ and {\em strict} if $l(dy) < l(y)$ for all chains $y \neq 0.$ It is straightforward to see that $CF(H;J)$ with $\cl{A}_H$ is a strict filtered complex over $\Lambda_{\bK}.$

Furthermore, for $a \in \R \setminus \spec(H)$ the subspace $CF(H,J)^{<a}$ spanned by all generators $\bar{x}$ with $\cA_{H}(\bar{x}) < a$ forms a subcomplex with respect to $d_{H;J},$ and its homology $HF(H)^{<a}$ does not depend on $J.$ Arguing up to $\epsilon,$ one can show that a suitable continuation map sends $HF(H)^{<a}$ to $HF(H')^{<a + \cE_{+}(H-H')},$ for \[\cE_{+}(F) = \int_{0}^{1} \max_M(F_t)\,dt.\]
It shall also be useful to define $\cE_{-}(F) = \cE_{+}(-F),\; \cE(F) = \cE_{+}(F) + \cE_{-}(F).$ Moreover, for an admissible action window, that is an interval $I=(a,b),$ $a<b,$ $a,b \in \R\setminus \spec(H),$ we define the Floer homology of $H$ in this window $HF^*(H)^I$ as the homology of the quotient complex \[CF^*(H)^I = CF^*(H)^{<b}/ CF^*(H)^{<a}.\]

Finally, one can show that for each $a \in \R,$ $HF(H)^{<a}$ as well as $HF(H)^I$ for an admissible action window, depends only on the class $\til{\phi}_H$ of the path $\{\phi^t_H\}_{t \in [0,1]}$ in the universal cover $\til{\Ham}(M,\om)$ of the Hamiltonian group of $M.$ When we wish to emphasize the coefficient field in the definition of Floer homology, we write $HF(H,\Lambda_{\bK})^{<a}.$ We shall mostly work with $\bK=\bF_2$ and $\bK = \K = \bF_2[h^{-1},h]],$ for a formal variable $h,$ and write $\Lambda = \Lambda_{\bK}$ if the choice of $\bK$ is clear.


In the case when $H$ is degenerate, we consider a perturbation $\cD = (K^H,J^H),$ with $K^H \in \cH,$ such that $H^{\cD} = H \# K^{H}$ is non-degenerate, and $J^H$ is generic with respect to $H^{\cD},$ and define the complex $CF(H;\cD) = CF(H^{\cD};J^H)$ generated by $\til{\cO}(H;\cD) = \til{\cO}(H^{\cD}),$ and filtered by the action functional $\cA_{H;\cD} = \cA_{H^{\cD}}.$ An admissible action window $I=(a,b)$ for $H,$ remains admissible for all $K^H$ sufficiently $C^2$-small, and the associated homology groups $HF(H;\cD)^I$ are canonically isomorphic for all $K^H$ sufficiently $C^2$-small. Hence $HF(H)^I$ is defined as the colimit of the associated indiscrete groupoid.

\subsection{Spectral invariants.}\label{subsec:spec}

Given a filtered complex $(C,\cA),$ to each homology class $\alpha \in H(C)$, denoting by $H(C)^{<a} = H(C^{<a}),$ $C^{<a} = \cA^{-1} (-\infty,a),$ we define
a spectral invariant by \[c(\alpha, (C, \cA)) = \inf\{a \in \R\,|\, \alpha \in \ima(H(C)^{<a} \to H(C)) \} \in \R \cup \{-\infty\}.\] For $(C,\cA) = (CF(H;\cD),\cA_{H ;\cD})$ we denote $c(\alpha, H; \cD) = c(\alpha,(C,\cA)).$ Furthermore, one can obtain classes $\alpha$ in the Hamiltonian Floer homology by the PSS isomorphism. This lets us define spectral invariants by: 
\[c(\alpha_M, H; \cD) = c(PSS(\alpha_M),(CF(H;\cD),\cA_{H;\cD})),\] for $\alpha_M \in QH(M).$ 
From the definition it is clear that the spectral invariants do not depend on the almost complex structure term in $\cD.$ Moreover, if $H$ is non-degenerate, we may choose the Hamiltonian term in $\cD$ to vanish identically, and denote the resulting invariants by $c(-,H).$ Moreover, by \cite[Section 5.4]{BiranCorneaRigidityUniruling} spectral invariants remain the same under extension of coefficients, hence below we do not have to specify the Novikov field $\Lambda$ that we work over. Spectral invariants enjoy numerous useful properties that hold for rational symplectic manifolds, the relevant ones of which we summarize below:

\begin{enumerate}
	\item {\em spectrality:} for each $\alpha_M \in QH(M) \setminus \{0\},$ and $H \in \cH,$ \[c(\alpha_M, H) \in \Spec(H).\]
	\item {\em non-Archimedean property:} $c(-,H;\cD)$ is {a non-Archimedean filtration function on $QH(M),$ as a module over the Novikov field $\Lambda$ with its natural valuation. }
	\item {\em continuity:} for each $\alpha_M \in QH(M) \setminus \{0\},$ and $F,G \in \cH,$
	\[|c(\alpha_M,F) - c(\alpha_M,G)| \leq \cE(F-G),\] 
	\item {\em triangle inquequality:} for each $\alpha_M,\alpha'_M \in QH(M),$ and $F,G \in \cH,$
	\[c(\alpha_M \ast \alpha'_M,F\#G) \leq c(\alpha_M ,F) + c(\alpha'_M ,G),\]
	\item {\em invariance:} $c(\alpha_M, H)$ depends only on the element $\til{\phi} = [\{\phi^t_H\}_{t \in [0,1]}]$ in the universal cover $\til{\Ham}(M,\om)$ generated by $H.$
	
\end{enumerate}

We remark that by the continuity property, the spectral invariants are indeed defined for all $H \in \cH$ and all the properties above apply in this generality. Finally, by the invariance property, we shall consider the spectral invariants as functions on $\til{\Ham}(M,\om),$ and shall sometimes denote for brevity by $\cl A_{\til{\phi}},$ the action functional $\cl A_H$ for a certain Hamiltonian generating $\til{\phi}.$

Below, we will sometimes be using the ground field $\K = \F_2[h^{-1},h]],$ for a formal variable $h,$ for the Novikov field. This is the field of fractions of the ring $\cl{L} = \F_2[[h]].$ In this case, we require the following observation. 

\begin{lma}\label{lma: action of def}
Let $F \in \cl{H}$ be a non-degenerate Hamiltonian. Consider elements $P \in CF(F,J; \Lambda_{\F_2}) \subset CF(F,J; \Lambda_{\K}),$ and $Q \in CF(F,J; \Lambda_{\cl{L}}) \subset CF(F,J; \Lambda_{\K}).$ Then \[\cl{A}_F(P+hQ) = \max\{ \cl{A}_F(P), \cl{A}_F(Q)\}.\]
\end{lma}

The proof of this lemma is essentially immediate, since writing $P = \sum \lambda_j \ol{x}_j$ with $\lambda_j \in \Lambda_{\F_2}$ in the basis $\{\ol{x}_j\}$ and $h Q = \sum h \mu_j \ol{x}_j$ with $\mu_j \in \Lambda_{\cl{L}},$ we have \[ P+hQ = \sum (\lambda_j + h \mu_j) \ol{x}_j.\] The lemma now follows from \eqref{eq: action of sum} and \[\nu(\lambda_j+h \mu_j) = \min \{ \nu(\lambda_j), \nu(\mu_j)\}\] for each $j.$ We prove the latter statement for a given fixed $j.$ Writing $\lambda_j = \sum a_l T^{\sigma_l},$ $a_l \in \F_2,$ $\mu_j = \sum b_l T^{\sigma_l},$ $b_l \in \cl{L}$ with the sets $\{ \sigma_l > c\,|\, a_l \neq 0 \},$ $\{ \sigma_l > c\,|\, b_l \neq 0 \},$ finite for all $c \in \R,$ we get \[ \lambda_j + h \mu_j = \sum (a_l + h b_l) T^{c_l}.\] As $a_l + h b_l = 0$ if and only if $a_l =0$ and $b_l =0,$ because $a_l \in \F_2,$ $h b_l \in h \cl{L},$ the conclusion now follows.   

%

\subsection{Homotopy-canonical filtered complexes and local Floer homology}\label{sec: homological perturb}

The theme of this section is that the situation of filtered Hamiltonian Floer homology of $H \in \cl H$ with $\fix(\phi^1_H)$ finite is very similar to that of the non-degenerate case, once we allow a finite-dimensional graded vector space of generators to be supported at each $x \in \fix(\phi^1_H).$ We first prove the following result, which is a chain-level enhancement of \cite[Lemma 2.2]{GG-pseudorotations}. Its proof below relies on homological perturbation techniques and constitutes a Novikov-field version of the canonical $\Lambda^0$-complexes from \cite{S-HZ}. While we formulate it for monotone symplectic manifolds to simplify notation, it extends almost verbatim to the case of rational symplectic manifolds.


\begin{thm}\label{thm: homological perturbation complex}
Let $(M,\om)$ be a closed monotone symplectic manifold. Consider the class $\til{\phi} \in \til{\Ham}(M,\om)$ of the Hamiltonian flow $\{\phi^t_H\}_{t \in [0,1]}$ of $H\in \cl H,$ with $\fix(\phi^1_H)$ finite. For a ground field $\bK,$ there is a homotopy-canonical complex $C(H)$ over the Novikov field $\Lambda_{\bK}$  on the action-completion of \[\oplus HF^{\loc}_{\ast}(\til{\phi},\ol{x})\] the sum running over all capped one-periodic orbits $\ol{x} \in \til{\cO}(H),$ that is free and graded over $\Lambda_{\bK},$ and is strict, that is $\cl{A}_H (d_{H} y) < \cl{A}(y)$ for all $y \in C(H),$ with respect to the non-Archimedean action-filtration $\cl{A}_H$ on $C(H)$ defined as follows: \begin{equation}\label{eq: action of sum loc} \cl A_H( \sum \lambda_j y_j ) = \max\{ -\nu(\lambda_j) + \cl A_H(y_j) \},\end{equation} \[\cl{A}_H(y_j) = \cl{A}_H(\ol{x}_{i(j)})\] for a $\Lambda$-basis $\{y_j\}$ of $C(H)$ given by a basis $\{y_i\;|\; i(j) = i\}$ of $HF^{\loc}_*(\til{\phi},\ol{x}_{i(j)}),$ for a choice of lifts $\{\ol{x}_i\}$ of $\fix(\phi) = \{x_i\}$ to capped orbits $\til{\cl{O}}(H).$  Furthermore the filtered homology $HF(H)^{<a}$ is given by $HF(C(H)^{<a}),$ $C(H)^{<a} = (\cl{A}_H)^{-1}\,(-\infty,a),$ for all $a \in \R \setminus \spec(H).$ In particular $HF(H) = H(C(H),d_{H}) \cong QH(M;\Lambda_{\bK}).$
\end{thm}

Moreover all the Floer-theoretical operations that we consider extend naturally to this chain-level setting. We shall only need one instance of this naturality formulated in Proposition \ref{prop: homological perturbation PSS} below. We now outline a proof of Theorem \ref{thm: homological perturbation complex}.



\begin{proof}[Proof of Theorem \ref{thm: homological perturbation complex}]
We let $H_1$ be a sufficiently $C^2$-small perturbation of a Hamiltonian $H.$ It generates a Hamiltonian diffeomorphism $\phi_1,$ whose contractible fixed points separate into clusters $\fix(\phi_1,x) \subset \fix(\phi_1)$ of fixed points of $\phi_1$ near $x \in \fix(\phi).$ Furthermore the corresponding capped periodic orbits $\til{\cO}(H_1)$ split into clusters $\til{\cO}(H_1, \ol{x})$ of orbits near $\ol{x} \in \til{\cO}(H)$ in $\til{\cl{L}}^{\min}_{pt} M,$ in such a way that the evaluation of a periodic orbit at $0$ is an isomorphism of sets $\til{\cO}(H_1, \ol{x}) \to \fix(\phi_1,x)$ for each capping $\ol{x}$ of $\alpha(x,H),$ and for each $A \in (2N_M) \Z,$ $\til{\cO}(H_1, \ol{x} \# A) = \til{\cO}(H_1, \ol{x}) \# A.$ Following \cite{SalamonZehnder, GG-local-gap} we observe that the elements of $\til{\cO}(H_1, \ol{x})$ and the Floer trajectories between them form a complex $CF_*^{\loc}(H_1,\ol{x}),$ with differential $d^{\loc,\ol{x}}$ whose homology $HF_*^{\loc}(\til{\phi},\ol{x})$ depends only on the class $\til{\phi} \in \til{\Ham}(M,\om)$ of the path $\{ \phi^t_H \}_{t \in [0,1]}$ (in particular it does not depend on the choice of $H_1$ given that it is sufficiently close to $H$). Furthermore, by \cite{GG-hyperbolic, GG-revisited, McLean-geodesics}, as well as \cite{FukayaOno1, SZhao-pants}, there is a crossing energy $2\epsilon_0 > 0$ depending only on $H$ and $x \in \fix(\phi),$ such that each $H_1$ Floer trajectory $u(s,t)$ asymptotic to $\alpha(x_1,H_1)$ as $s \to {-\infty}$ where $x_1 \in \fix(\phi_1,x)$ is either contained in a small isolating neighborhood $U_x$ of $\alpha(x,H),$ in $S^1 \times M,$ and hence connects two elements of $\til{\cO}(H_1, \ol{x})$ for each capping $\ol{x}$ of $\alpha(x,H),$ or has energy $E(u) = E_{H_1}(u) \geq 2\epsilon_0.$  Finally, the actions $\cl{A}_{H_1}(\ol{x}_1)$ and indices $CZ(H_1,\ol{x}_1)$ for $\ol{x}_1 \in \til{\cO}(H_1, \ol{x})$ satisfy \[ |\cl{A}_{H_1}(\ol{x}_1) - \cl{A}_H(\ol{x})| < \delta \ll \eps_0,\] \[ |CZ(H_1,\ol{x}_1) - \Delta(H,\ol{x})| \leq n.\]

This implies that the Floer differential of the Floer complex $CF(H_1;\cl{D})$ splits as \begin{equation}\label{eq: d split} d = d^{\loc} + D \end{equation} with $d^{\loc} = \oplus_{\ol{x} \in \cl{O}(H)}\, d^{\loc,\ol{x}}$ suitably completed, and $\cl{A}_{H_1}(dy) \leq \cl{A}_{H_1}(y) - \eps_0$ for all chains $y \in CF(H_1;\cl{D}).$ Observe that $d^{\loc}$ is a differential of $CF(H_1;\cl{D})$ as a $\Lambda$-module.

Now choose subspaces $X_{\ast}^{\loc}(\ol{x}) \subset \ker(d^{\loc,\ol{x}}) \subset CF^{\loc}_{\ast}(H_1,\ol{x})$ such that \[X_{\ast}^{\loc}(\ol{x} \# (2N_M) k) = {q^{k}} X_{\ast}^{\loc}(\ol{x})\] for all $k \in \Z,$ and the inclusion \[\iota_{\ol{x}}: (X_{\ast}^{\loc}(\ol{x}),0) \subset (CF^{\loc}_{\ast}(H_1,\ol{x}), d^{\loc,\ol{x}}) \] is a quasi-isomorphism. Choose projections \[\pi_{\ol{x}}: (CF^{\loc}_{\ast}(H_1,\ol{x}), d^{\loc,\ol{x}}) \to (X_{\ast}^{\loc}(\ol{x}),0)\] similarly compatible with the Novikov action, such that \[\pi_{\ol{x}} \circ \iota_{\ol{x}} = \id_{X_{\ast}^{\loc}(\ol{x})},\] \[\iota_{\ol{x}} \circ \pi_{\ol{x}} = \id_{CF^{\loc}_{\ast}(H_1,\ol{x})} + d^{\loc,\ol{x}}\Theta_{\ol{x}} + \Theta_{\ol{x}}d^{\loc,\ol{x}},\] for homotopies \[\Theta_{\ol{x}}: CF^{\loc}_{\ast}(H_1,\ol{x}) \to CF^{\loc}_{\ast+1}(H_1,\ol{x})\] again compatible with the Novikov action, and satisfying ${\Theta_{\ol{x}}}^2 = 0$ and $\cl{A}(\Theta_{\ol{x}} y) < \cl{A}(y) + \delta$ for all chains $y.$ We refer to \cite[Section 6]{S-HZ} for a detailed construction of such $\Theta_{\ol{x}}, \iota_{\ol{x}}, \pi_{\ol{x}}$ in a similar setting, as the two local settings can be identified by fixing one capping $\ol{x}_0,$ making the choices for it, and then extending to all other cappings $\ol{x}$ by compatibility with the Novikov action. 

The homological perturbation formulae now yield a differential $d_{H}$ on the $\Lambda$-module $C_*(H)$ given by completing $\oplus X^{\loc}_{\ast}$ with respect to the filtration $\cl{A}_{H_1},$ as well as injection $\ol{\iota}: C_*(H) \to CF_*(H_1;\cl{D}),$ projection $\ol{\pi}:CF_*(H_1;\cl{D}) \to C_*(H),$ and homotopy $\ol{\Theta}: CF_*(H_1;\cl{D}) \to CF_{*+1}(H_1;\cl{D})$ satisfying \[\ol{\pi} \circ \ol{\iota} = \id_{C(H)},\] \[\ol{\iota} \circ \ol{\pi} = \id_{CF(H_1\; \cl{D})} + d\ol{\Theta} + \ol{\Theta} d,\] \[\ol{\Theta}^2 = 0\] and \[\cl{A}_{H_1}(\ol{\Theta} y) < \cl{A}_{H_1}(y) + \delta\] for all chains $y.$ Furthermore \[ \cl{A}_{H_1}(d_H(y)) \leq \cl{A}_{H_1}(y) - 2\eps_0 + \delta \] for all chains $y \in C(H).$ In particular for each $y,z \in X^{\loc}_{\ast}(\ol{x})$ the coefficient $\left< d_H(y), z \right>$ vanishes. Now define the filtration $\cA_{H}$ on $C(H)$ by setting \[\cA_H(y) : = \cA_H(\ol{x})\] for all $y \in X^{\loc}_{\ast}(\ol{x}),$ and extending it naturally to the completion. This definition is easily seen to coincide with the description in Theorem \ref{thm: homological perturbation complex}. Then \[ |\cl{A}_{H_1}(y) - \cl{A}_H(y)| < \delta\] for all chains $y \in C(H) \setminus \{0\}.$ In particular $(C(H),d_H)$ is strict with respect to $\cl{A}_H.$ Finally, it is easy to see by a filtration argument that the complexes $C(H)$ obtained from different sufficiently small perturbations $H_1$ of $H$ are all filtered-isomorphic: indeed the continuation maps between them yield chain maps that induce isomorphisms on the local homology groups, and split similarly to $d$ in \eqref{eq: d split}. Hence the homological perturbation formulae produce maps on the complexes $C(H)$ that are of the form $\gamma^{\loc} + \Gamma,$ where $\cl{A}_H(\Gamma(y)) \leq \cl{A}_H(y) - \eps_0$ for all chains $y,$ and $\gamma^{\loc}$ is a linear isomorphism preserving the filtration. The conclusion follows.


Now we use monotonicity, or in fact rationality, to prove that the filtered complex $(C(H),\cl{A}_H)$ calculates the filtered Floer homology of $H.$ By rationality, and $\phi$ having isolated contractible fixed points, we obtain that there is $\eps_1 > 0$ such that \[|\cA_H(\ol{x}) - \cA_{H}(\ol{y})| \geq \eps_1\] for each two capped orbits $\ol{x}, \ol{y}$ of $H$ with distinct actions. Now for $a \in \R \setminus \spec(H),$ by the independence of $C(H)$ up to filtered isomorphism on $H_1,$ we choose $H_1$ such that $d(a, \spec(H)) > \delta$ and hence $\cl{A}_{H_1}(\cl{O}(H_1,\ol{x}))$ is contained either in $(-\infty,a)$ or in $(a,\infty)$ for all $\ol{x} \in \cl{O}(H).$ Therefore $HF(H)^{<a} = HF(H_1)^{<a}$ by definition, and $HF(H_1)^{<a} = H(CF(H_1)^{<a}) = H(CF(H)^{<a})$ by construction. 

Finally, it is not hard to see that the construction of $C(H)$ does not depend on the choices made in the proof up to filtration-preserving chain-homotopies.
\end{proof}


Moreover, the proof of Lemma \ref{lma: action of def} adapts tautologically to prove the following.

\begin{lma}\label{lma: action of def 2}
	Let $F \in \cl{H}$ be a Hamiltonian with $\fix(\phi^1_F)$ finite. Let $P$ be in $C(F; \Lambda_{\F_2}) \subset CF(F; \Lambda_{\K}).$ Let $Q$ be in $C(F; \Lambda_{\cl{L}}) \subset C(F; \Lambda_{\K}).$ Then \[\cl{A}_F(P+hQ) = \max\{ \cl{A}_F(P), \cl{A}_F(Q)\}.\]
\end{lma}

\section{Proof of Theorem \ref{thm: uniruled}}

Consider the filtered Floer homology $HF(\til{\phi},\Lambda)^{<c},$ $\Lambda = \Lambda_{\F_2},$ where $c \in \R \setminus \spec(\til{\phi}),$ and the filtered $\zt$-equivariant Tate Floer homology $\hat{HF}(\til{\phi}^2)^{<c}$ for level $c \in \R \setminus \spec(\til{\phi}^2).$ The spectral invariant $c(a,\til{\phi}),$ for $a \in QH(M) \setminus \{0\},$ is defined as \[ c(a,\til{\phi}) = \inf \{ c \in \R \setminus \spec(\til{\phi})\;|\; PSS(a) \in \ima \left(HF(\til{\phi})^{<c} \to HF(\til{\phi}) \right) \}.\] For $a \in QH(M,\Lambda_{\K})$ let the $\zt$-equivariant Tate spectral invariant  $\wh{c}(a, \til{\phi}^2)$ be \[ \hat{c}(a,\til{\phi}^2) = \inf \{ c \in \R \setminus \spec(\til{\phi}^2)\;|\; PSS_{\zt}(a) \in \ima \left(\hat{HF}(\til{\phi}^2)^{<c} \to \hat{HF}(\til{\phi}^2) \right) \},\] where $PSS_{\zt}$ is the equivariant PSS isomorphism introduced in \cite{Wilkins-PSS}.

By the construction of Seidel \cite{Seidel-pants} and Wilkins \cite{Wilkins-PSS}, combined with the action estimates as in \cite{SZhao-pants} for example, the equivariant pair-of-pants product precomposed with Kaledin's quasi-Frobenius map yields an injective map, for $\K = \F_2[h^{-1},h]],$ and $c\in \R,$ 
\[ \cP: HF(\til{\phi}, \Lambda_{\K})^{<c} \to \wh{HF}_{\zt}(\til{\phi}^2)^{<2c}.\] 

Furthermore, by \cite{Wilkins-PSS}, this map commutes with vertical maps \[HF(\til{\phi}, \Lambda_{\K})^{<c} \to HF(\til{\phi}, \Lambda_{\K}) \xleftarrow{PSS} QH(M,\Lambda_{\K}),\] \[\wh{HF}_{\zt}(\til{\phi}^2)^{<2c} \to \wh{HF}_{\zt}(\til{\phi}^2) \xleftarrow{PSS_{\zt}} QH(M,\Lambda_{\K}),\] where $PSS,$ $PSS_{\zt}$ may also be replaced by the inverse $PSS$ maps, and the horizontal map: \[\cl{QS}: QH(M,\Lambda_{\K}) \to QH(M,\Lambda_{\K})\] given by the quantum Steenrod square from \cite{Wilkins}. From this commutativity, we immediately obtain the inequality \begin{equation}\label{eq: doubling ineq} \hat{c}(\cl{QS}(y),\til{\phi}^2) \leq 2 c(y, \til{\phi}).\end{equation} 

We shall require one feature of the equivariant PSS map $PSS_{\zt}$: for each chain $z \in CM(f,\Lambda_{\K})$ in the Morse complex of a Morse function $f$ on $M,$ we have the identity \[ PSS_{\zt}(z) = PSS(z) + h R(z),\] where $h R(z) = \sum_{j \geq 1} h^j PSS_{\zt,j}(z)$ is a collection of terms of higher order in $h.$ We require the following analogue of this statement in the case of isolated, but possibly degenerate, contractible fixed points. It is deduced from its non-degerate version by adapting the proof of Theorem \ref{thm: homological perturbation complex} above to the equivariant Tate complex, as in \cite[Section 6]{S-HZ} for the $\Lambda^0$ case, and inducing maps on the homotopy-canonical complexes from the PSS and equivariant PSS maps respectively. 

\begin{prop}\label{prop: homological perturbation PSS}
	Consider $H \in \cl{H}$ with $\fix(\phi^1_{H^{(2)}})$ finite. Then the PSS isomorphism construction induces quasi-isomorphisms $PSS: C(M; \cD,\Lambda_{\bK}) \to C(H^{(2)})$ and $PSS_{\zt}: C(M; \cD,\Lambda_{\K}) \to \wh{C}(H^{(2)}) = C(H^{(2)}) \otimes_{\Lambda_{\bK}} \Lambda_{\K}.$ Furthermore, for compatible choices of auxiliary data we have the relation: \[ PSS_{\zt} = PSS + h R \] for  $h R  = \sum_{j \geq 1} h^j PSS_{\zt,j}$ is a collection of terms of higher order in $h.$ In fact $PSS_{\zt,j}$ is defined over $\Lambda_{\bK}.$  
\end{prop}


Furthermore, we have the following key technical result specific to pseudo-rotations.

\begin{prop}\label{prop: unique carriers}
For each class $z \in QH(M,\Lambda_{\K}),$ and $H \in \cl{H},$ with $\phi^1_H$ a pseudo-rotation, the chain representatives of $PSS(z)$ in $C(H)$ and $PSS_{\zt}(z)$ in $\wh{C}_{\zt}(H^{(2)})$ are unique and coincide and their action levels coincide with their spectral invariants: $c(z, \til{\phi})$ and $\wh{c}(z, \til{\phi}^2).$
\end{prop}

\begin{rmk}
We shall use Proposition \ref{prop: unique carriers} to relate $c([pt], H^{(2)})$ and $\wh{c}([pt], H^{(2)}).$ 
\end{rmk}


\begin{proof}[Proof of Proposition \ref{prop: unique carriers}]

First of all, by comparing the dimensions of $C(H)$ and $QH(M, \Lambda_{\F_2})$ over $\Lambda_{\F_2}$ condition \ref{cond: sum of Betti} of the pseudo-rotation $\phi^1_H$ implies that the differential $d_H$ on $C(H)$ vanishes. Hence each homology class $a \in C(H) \setminus \{0\}$ has a unique representative. This is in particular true for $a = PSS(z)$ for $z \in QH(M,\Lambda_{\F_2}),$ and the compatibility of $C(H)$ with the action filtration shows that $\cA_{H}(PSS(z)) = c(z, H).$ The same holds for coefficients in the field extension $\Lambda_{\K}$ of $\Lambda_{\F_2}.$


A similar, yet slightly more complicated, argument applies to the equivariant Tate case. Indeed, by condition \ref{cond: perfect} of a pseudo-rotation, the homotopy-canonical equivariant Tate complex $\wh{C}(H^{(2)})$ is given as the action-completion of \[\oplus_{\ol{x} \in \cl{O}(H^{(2)})} HF^{\loc}(H^{(2)},\ol{x}) \otimes {\K},\] where each such capped $\ol{x}$ is a recapping of an iterated capped orbit $\ol{y}^{(2)}$ of $H,$ and its differential, which is $\Lambda_{\K}$-linear, is given by \[\wh{d}_{H^{(2)}} = \hat{d}_{H^{(2)}}^{\loc} + \wh{D}_{H^{(2)}},\] where $\cl{A}_{H^{(2)}}(\wh{D}_{H^{(2)}}(y)) \leq \cl{A}_{H^{(2)}}(y) - \eps_0$ for all chains $y.$ We claim that $\wh{d}_{H^{(2)}} = 0.$ Indeed, by the equivariant PSS isomorphism, $(\wh{C}(H^{(2)}), \wh{d}_{H^{(2)}})$ is quasi-isomorphic as a $\Lambda_{\K}$-module to $QH(M, \Lambda_{\K}).$ However, the dimension of $\wh{C}(H^{(2)})$ over $\Lambda_{\K}$ is given by $N(\phi^2,\F_2) = \dim_{\K} QH(M, \Lambda_{\K})$ by condition \ref{cond: sum of Betti} of a pseudo-rotation. This finishes the proof.
\end{proof}

\begin{rmk}\label{rmk: doubling}
	In fact, for {\pr}s \eqref{eq: doubling ineq} becomes an equality, that is for $y \in QH(M, \Lambda_{\K}),$ we have the identity of spectral invariants: \[\hat{c}(\cl{QS}(y), \til{\phi}^2) = 2c(y, \til{\phi}).\] Indeed, consider $y \in QH(M),$ and let $c = c(y, \til{\phi})$ be its critical level. Since the equivariant pair of pants product is an isomorphism on local Floer homology (see \cite{SZhao-pants, S-HZ}), it gives an isomorphism of homologies in action windows \[HF(\til{\phi}, \Lambda_{\K})^{(c-\epsilon,c+\epsilon)} \to \wh{HF}_{\zt}(\til{\phi}^2)^{(2c-2\epsilon, 2c+ 2\epsilon)} = {HF}(\til{\phi}^2,\Lambda_{\K})^{(2c-2\epsilon, 2c+ 2\epsilon)}\] for all $\epsilon$ sufficiently small. This shows that $2c$ is the critical level of $\cl{QS}(y)$ for $\til{\phi^2}.$ Indeed, otherwise the non-zero image of the chain representative of $y$ in the leftmost homology would go to zero in the rightmost homology, for $\epsilon$ sufficiently small.
\end{rmk}


\begin{proof}[Proof of Theorem \ref{thm: pseudo non eq uni}]
 From estimate \eqref{eq: doubling ineq} we obtain the bound \[ \hat{c}(\cl{QS}([pt]), \til{\phi}^2) \leq 2 c([pt], \til{\phi}).\] However, \[\hat{c}(\cl{QS}([pt]), \til{\phi}^2) = \hat{c}(h^{2n}[pt], \til{\phi}^2) = \hat{c}([pt], \til{\phi}^2).\] Identifying between the class $[pt]$ and its chain level representative, by choosing a Morse function on $M$ with unique minimum, which represents the point class, by Proposition \ref{prop: unique carriers} the following identities hold \[\hat{c}([pt], \til{\phi}^2) = \cA_{\til{\phi}^2}(PSS_{\zt}([pt])),\] \[ {c}([pt], \til{\phi}^2) = \cA_{\til{\phi}^2}(PSS([pt])).\] Furthermore, by Proposition \ref{prop: homological perturbation PSS} combined with Lemma \ref{lma: action of def 2}, \[\hat{c}([pt], \til{\phi}^2) \geq c([pt],\til{\phi}^2).\] This finishes the proof.  
\end{proof}


\bibliographystyle{abbrv}
\bibliography{bibliographyPRQS}

\end{document}